\documentclass[11pt]{amsart}

\usepackage{amssymb}
\usepackage{enumerate}
\usepackage{mathrsfs}
\usepackage{relsize}

\makeatletter
\def\imod#1{\allowbreak\mkern10mu({\operator@font mod}\,\,#1)}
\makeatother

\newtheorem{theorem}{Theorem}[section]

\newtheorem{prop}[theorem]{Proposition}
\newtheorem{lemma}[theorem]{Lemma}

\theoremstyle{definition}

\theoremstyle{remark}
\newtheorem{remark}[theorem]{Remark}
\newtheorem{remarks}[theorem]{Remarks}
\theoremstyle{remark}

\numberwithin{equation}{section}

    \DeclareMathOperator{\iso}{iso}

    \DeclareMathOperator{\Mod}{Mod}

     \DeclareMathOperator{\II}{II}
    \DeclareMathOperator{\vN}{vN}
    \DeclareMathOperator{\III}{III}
    \DeclareMathOperator{\I}{I}

    \DeclareMathOperator{\itpf1}{ITPFI}
    
    \DeclareMathOperator{\Tr}{Tr}

\def\R{{\mathbb R}}
\def\C{{\mathbb C}}

\def\N{{\mathbb N}}
\def\Z{{\mathbb Z}}

\def\Q{{\mathbb Q}}

\begin{document}

\title[Turbulence and Araki-Woods factors]{Turbulence and Araki-Woods factors}

\author{Rom\'an Sasyk}

\address{
Instituto Argentino  de Matem\'aticas-CONICET\\
Saavedra 15, Piso 3 (1083), Buenos Aires, Argentina}

\email{rsasyk@conicet.gov.ar}

\author{Asger T\"ornquist}
\address{Kurt G\"odel Research Center, University of Vienna, W\"ahringer Strasse 25, 1090 Vienna, Austria}
\email{asger@logic.univie.ac.at}

\thanks{\emph{Acknowledgments.} Research for this paper was mainly carried out at the Hausdorff Institute for Mathematics in Bonn during the `Rigidity' programme in the fall semester of 2009. We wish to thank the Hausdorff Institute for kind hospitality and support. A. T\"ornquist was also supported in part through Austrian Science Foundation FWF grant no. P19375-N18 and a Marie Curie grant no. 249167 from the European Union.}




\maketitle

\begin{abstract}
Using Baire category techniques we prove that Araki-Woods factors
are not classifiable by countable structures. As a result, we obtain
a far reaching strengthening as well as a new proof of the well-known theorem of Woods
that the isomorphism problem for $\itpf1$ factors is not smooth.
We derive as a consequence that the odometer
actions of $\Z$ that preserve the measure class of a finite
non-atomic product measure are not classifiable up to orbit
equivalence by countable structures.
\end{abstract}

\section{Introduction}

The present paper continues a line of research into the structure of
the isomorphism relation for separable von Neumann algebras using
techniques from descriptive set theory, which was initiated in
\cite{sato09a} and \cite{sato09b}.

The central notion from descriptive set theory relevant to this
paper is that of \emph{Borel reducibility}. Recall that if $E$ and
$F$ are equivalence relations on Polish spaces $X$ and $Y$,
respectively, we say that $E$ is \emph{Borel reducible} to $F$ if
there is a Borel function $f:X\to Y$ such that
$$
(\forall x,x'\in X) x E x'\iff f(x) F f(x'),
$$
and if this is the case we write $E\leq_B F$. Borel reducibility is
a notion of relative complexity of equivalence relations and the
isomorphism problems they pose, and the statement $E\leq_B F$ is
interpreted as saying that the points of $X$ are classifiable up to
$E$-equivalence by a Borel assignment of complete invariants that
are $F$-equivalence classes. The requirement that $f$ be Borel is a
natural restriction to ensure that the invariants are assigned in a
reasonably definable way. Without a definability condition on the
function $f$ reducibility would amount only to a consideration of
the cardinality of the quotient spaces $X/E$ and $Y/F$.

In \cite{sato09a} it was shown that the isomorphism relation in all
the natural classes of separable von Neumann factors, $\II_1$,
$\II_\infty$, $\III_\lambda$, $(0\leq\lambda\leq 1)$ do not admit a
classification by countable structures. That is, if $\mathcal L$ is
a countable language and $\Mod(\mathcal L)$ is the natural Polish
space of countable $\mathcal L$-structures (see \cite[\S
2.3]{hjorth00}), then there is no Borel reduction of the isomorphism
relation of von Neumann factors of any fixed type to the isomorphism
relation $\simeq^{\Mod(\mathcal L)}$ in $\Mod(\mathcal L)$. This in
particular implies that there is no Borel assignment of countable
groups, graphs, fields or orderings as complete invariants for the
isomorphism problem for factors.

Recently, Kerr, Li and Pichot in \cite{kelipi} obtained
several non-classification results along the same lines exhibited here but
for the automorphism groups of finite factors. For instance they showed that
the conjugacy relation for the trace-preserving free weakly mixing
actions of discrete groups on a $\II_1$ factor is not classifiable by countable structures.

These types of results are much stronger than the classical
smooth/non-smooth dichotomy, since they give specific information
about the complexity of the kind of invariant that can be used in a
complete classification. They are also stronger than the traditional
smooth/non-smooth dichotomy for equivalence relation, since, for
instance, isomorphism of countable groups is \emph{not smooth}, yet
in many cases countable groups are reasonable invariants.

The earliest non-smoothness result for the isomorphism relation of
von Neumann algebras is Woods' Theorem \cite{woods73}, which asserts
that the isomorphisms relation for $\itpf1_2$ factors is not smooth.
Recall that a von Neumann algebra $M$ is called an Araki-Woods
factor or an $\itpf1$ factor (short for {\it infinite tensor product
of factors of type }$\I$) if it is of the form
$$
M=\bigotimes_{k=1}^\infty (M_{n_k}(\C),\phi_k)
$$
where $M_{n_k}(\C)$ denotes the algebra of $n_k\times n_k$ matrices
and the $\phi_k$ are faithful normal states. In the case when
$n_k=2$ for all $k$, the factor $M$ is called $\itpf1_2$. In this
paper we will show:

\begin{theorem}
The isomorphism relation for $\itpf1_2$ factors is not classifiable
by countable structures.\label{mainthm}
\end{theorem}

This solves a problem posed in \cite{sato09b}, and provides a
strengthening and a new proof of Woods' Theorem. It also provides a
new and more direct proof that the isomorphism relation for
injective type $\III_0$ factors is not classifiable by countable
structures, a result proven in \cite{sato09a} using Krieger's
Theorem regarding the duality between flows and injective factors,
\cite{krieger76}.

\medskip

The $\itpf1$ factors constructed in the proof of Theorem
\ref{mainthm} correspond to group-measure space factors constructed
from the measure-class preserving odometer actions of $\Z$ on
$\{0,1\}^\N$, when $\{0,1\}^\N$ is equipped with a finite product
measure. Therefore we obtain the following interesting corollary:

\begin{theorem}
The odometer actions of $\Z$ on $\{0,1\}^\N$ preserving the measure
class of a finite non-atomic ergodic product measure are not classifiable, up to orbit
equivalence, by countable structures.
\end{theorem}

This stands in contrast to Dye's Theorem for probability measure
preserving actions, and may be compared with the theorem of Ioana,
Kechris, Tsankov and Epstein in \cite{iokets09} on the
non-classifiability up to orbit equivalence of probability measure
preserving ergodic actions of a countable non-amenable group.

\section{A turbulence lemma}

In this section we establish a general lemma which shows that a wide
class of natural actions are turbulent, in the sense of
\cite{hjorth00}. Recall that if $G$ is a Polish group acting
continuously on a Polish space $X$, then the action is said to be
turbulent if the following holds\footnote{Strictly speaking, Hjorth
required a turbulent action to have dense, meagre orbits, but we
keep those requirements separate.}:

For all $x,y\in X$, all open $U\subseteq X$ with $x\in U$ and all
open $V\subseteq G$ containing the identity, there is $y_0\in U$ in
the $G$-orbit of $y$, such that for all neighbourhoods $U_0$ of
$y_0$ there is a finite sequence $x_i\in U$, $(0\leq i\leq n)$ with
$x_0=x$ and a sequence $g_i\in V$,  $(0\leq i< n)$, such that
$$
x_{i+1}=g_{i}\cdot x_i
$$
and $x_n\in U_0$.

Recall moreover that a {\it Fr\'echet space} is a completely
metrizable locally convex vector space (over $\R$ or $\C$.)

\begin{lemma}\label{turbulence}
Let $F$ be a separable Fr\'echet space and let $G\subseteq F$ be a
dense subgroup of the additive group $(F,+)$. Suppose $(G,+)$ has a
Polish group topology such that the inclusion map $i:G\to F$ is
continuous, and satisfies
\begin{quote}
\begin{enumerate}[$(*)$]
\item for all $g\in G$ and open $V\subseteq G$ with $0\in V$
there is $n\in\N$ such that $\frac 1 n g\in V$
\end{enumerate}
\end{quote}
(e.g. when $G$ itself is a Fr\'echet space.) Then either $G=F$ or
the action of $G$ on $F$ by addition,
$$
g\cdot x=g+x,
$$
is turbulent and has meagre dense classes.

In particular, if $(G,\|\cdot\|_{G})$ and $(F,\|\cdot\|_F)$ are
separable Banach spaces such that $G$ is a dense subspace of $F$ and
the inclusion map $i:G\to F$ is bounded, then either $G=F$ or the
action of $(G,+)$ on $F$ by addition is turbulent and has meagre
dense classes.\label{turbulence}
\end{lemma}

\begin{proof}
Note that $G$ is an analytic subset of $F$ since $i:G\to F$ is
continuous, and so it has the Baire Property in $F$ (see \cite[8.21
and 21.6]{kechris95}). So if $G\neq F$ then $G$ must be meagre in
$F$, since otherwise by Pettis' Theorem \cite[9.9]{kechris95} $G$
must contain a neighbourhood of the identity, and so $G=F$. Since $G$
is dense in $F$ it follows that all the $G$-orbits are meagre and
dense. So it suffices to show that the action of $G$ is turbulent.
For this, let $x\in F$ and let $U\subseteq F$ be a convex open
neighbourhood of $F$. Let $y\in U$, let $U_0$ be an open neighbourhood
of $y$ such that $y\in U_0\subseteq U$, and let $V\subseteq G$ be a
neighbourhood of $0$ in $G$. Since $G$ is dense in $F$ we may find
$g\in G$ such that $x+g\in U_0$. By assumption there is $n\in\N$
such that $\frac g n \in V$, and since $U$ is convex we have
$$
x,x+\frac 1 n g, x+ \frac 2 n g,\ldots, x+\frac {n-1} n g, x+g\in U,
$$
which shows that $G$ acts turbulently.
\end{proof}

\begin{remark} Many turbulence results found in the
literature are special instances of the above lemma. For instance,
let $(c_0,\|\cdot\|_\infty)$ denote the real Banach space of real
valued sequences that converge to zero, equipped with the sup-norm.
The elementary example \cite[3.23]{hjorth00} that $c_0$ acts
turbulently on $\R^\N$ by addition fits into this framework.
Moreover, condition $(*)$ may be replaced by the weaker condition
\begin{quote}
\begin{enumerate}[$(**)$]
\item for all $g\in G$ and $W,V\subseteq G$ open neighbourhoods
such that $g\in W$ and $0\in V$, there is $z\in V$ and $n\in\N$ such
that $n z\in W$.
\end{enumerate}
\end{quote}
in which case \cite[Proposition 3.25]{hjorth00} also follows from
the above.

The results in \cite{sofronidis05} also fall into this category.
Indeed, let $X$ be a locally compact  not compact Polish space. The
space $C(X,\R)=\{f:X\to\R \text{ continuous}\}$ is a separable
Fr\'echet space with the topology given by uniform convergence in
compact sets. $C_0(X,\R)=\{f\in C(X,\R): \lim_{x\to\infty} f(x)=0\}$
is a dense subspace of $C(X,\R)$ and it is Polish in the
topology given by uniform convergence. It follows from the previous
lemma that the natural action of $C_0(X,\R)$ on $C(X,\R)$ is
turbulent, has meagre classes and every class is dense. The
exponential map then gives \cite[Theorem 1.1]{sofronidis05}. In a
similar fashion one could also recover \cite[Theorem
1.2]{sofronidis05}.
\end{remark}

\subsection*{$S_\infty$-ergodicity and turbulence} Let $S_\infty$
denote the group of all permutations of $\N$. The importance of the
notion of turbulence comes from its relation to the notion of
\emph{$S_\infty$-ergodicity}. Recall that an equivalence relation
$E$ on a Polish space $X$ is said to be \emph{generically
$S_\infty$-ergodic} if whenever $S_\infty$ acts continuously on a
Polish space $Y$ giving rise to the orbit equivalence relation
$E_{S_\infty}^Y$, and $f:X\to Y$ is a Baire measurable function such
that
$$
(\forall x,x') xE x'\implies f(x) E_{S_\infty}^Y f(x')
$$
(i.e. $f$ is a homomorphism of equivalence relations), then there is
a single $S_\infty$-orbit $[y]_{E_{S_{\infty}}^Y}$ such that
$$
\{x\in X: f(x)\in [y]_{E_{S_{\infty}}^Y}\}
$$
is comeagre. Note that if $E\subseteq E'$, where $E'$ is also an
equivalence relation, and $E$ is generically $S_\infty$-ergodic,
then so is $E'$.

The fundamental Theorem in Hjorth's theory of turbulence is that if
a Polish group $G$ acts continuously and turbulently on a Polish
space $X$, then the associated orbit equivalence relation $E_G^X$ is
generically $S_\infty$-ergodic, see \cite[Theorem 3.18]{hjorth00}.
Since the isomorphism relation $\simeq^{\Mod(\mathcal L)}$ in the
Polish space $\Mod(\mathcal L)$ of countable $\mathcal L$-structures
is induced by a continuous $S_\infty$-action, turbulence provides an
obstruction to Borel reducibility of $E_G^X$ to
$\simeq^{\Mod(\mathcal L)}$ if the $G$-orbits are meagre. The
isomorphism relations for countable groups, graphs, fields,
orderings, etc., are special instances of $\simeq^{\Mod(\mathcal
L)}$ for appropriate choices of the language $\mathcal L$, and so
turbulence can be used to prove the impossibility of obtaining a
complete classification by a reasonable (i.e. Borel or Baire
measurable) assignment of such countable objects as invariants.

\section{Proof of the main theorem}

We will now define a family $(M_x)_{x\in c_0}$ of $\itpf1$ factors
parameterized by elements of $c_0$. The motivation behind the
definition comes from the results in \cite{gioskan}. For $j\in\N$
define $N_j=2^{j!}$, and for each $x\in c_0$ let
$$
l_j^x=\ln(2)j!e^{x(j)/j!}.
$$
Let $\phi_j^x$ be the state on $M_2(\C)$ given by
$$
\phi_j^x(a)=\frac{1}{1+e^{-l_j^x}}\Tr\Big(a\cdot\begin{bmatrix} 1 &
0\\0 & e^{-l_j^x}\end{bmatrix}\Big).
$$
Then we define $M_x$ to be the $\itpf1_2$ factor
$$
M_x=\bigotimes_{j=1}^\infty(M_2(\C),\phi_j^x)^{\otimes N_j}.
$$
In other words, $M_x$ is the $\itpf1_2$ factor with eigenvalue list
$(\lambda_n^x, 1-\lambda_n^x)_{n\in\N}$ where $\lambda_n^x$ is given
by
$$
\lambda_n^x=\frac 1 {1+e^{-l_j^x}}
$$
whenever $\sum_{i=1}^{j-1} N_i< n\leq \sum_{i=1}^{j} N_i$ for some
$j\in\N$. Since $\l_j^x\to \infty$  and $\sum_j
N_je^{-l_j^x}=\infty$, all the factors $M_x$ are of type $\III$,
\cite[III.4.6.6]{blackadar06}.
\medskip

Theorem \ref{mainthm} will be proved by showing that the family of
factors $(M_x)_{x\in c_0}$ is not classifiable up to isomorphism by
countable structures. An outline of the proof is as follows: First
we will show that the equivalence relation
$$
x \sim_{\iso} x'\iff M_x\text{ is isomorphic to } M_{x'}
$$
has meagre classes, thus showing that the family $(M_x)$ contains
uncountably many non-isomorphic factors. Then we will show that
there is a subgroup $G\subseteq c_0$ of the additive group $(c_0,+)$
that satisfies the hypothesis of Lemma \ref{turbulence}, and with
the additional property that
$$
M_{g+x}\simeq M_x
$$
for all $g\in G$ and $x\in c_0$, where $\simeq$ denotes isomorphism
of von Neumann algebras. From this fact it will be easy to deduce
that the equivalence relation $\sim_{\iso}$ is not classifiable by
countable structures. Finally we will show that the map $x\mapsto
M_x$ is Borel (in a precise way) and thus provides a Borel reduction
of $\sim_{\iso}$ to $\simeq$.

\bigskip
The main tool used to distinguish uncountably many
non-isomorphic elements of the family $(M_x)_{x\in c_0}$ is Connes'
invariant $T(M)$.
Recall that if $M$
is a von Neumann algebra with a faithful semifinite normal weight
$\varphi$, the Tomita-Takesaki theory associates to it a one
parameter group of automorphisms of $M$, the so-called {\it modular
automorphism group}. If $\sigma^\varphi_t$ denotes the modular
automorphism group of $(M,\varphi)$, the $T$-set of $M$ is the
additive subgroup of $\R$  defined by
$$
T(M)=\{t\in\R: \sigma^\varphi_t \text{ is an inner automorphism}\}.
$$
Even though $\sigma^\varphi_t$ depends on $\varphi$, Connes'
non-commutative Radon-Nikodym Theorem guarantees that $T(M)$ is
independent of the choice of the faithful semifinite normal weight
$\varphi$. The $T$-set is arguably the most important invariant
employed to distinguish injective type $\III_0$ factors
and it can be found already in Araki and Woods's seminal article \cite{arwo68}.
A thorough treatment of these important concepts that are at the heart
of the structural theory of factors of type $\III$ can be found in
\cite[5.3-5.5]{connesNCG}, \cite[III.3, III.4]{blackadar06} and
\cite{takesakiV2}. For the purpose of this article we will only need
the following Lemma.

\begin{lemma}\cite[Corollaire 1.3.9]{connesthesis}\label{T-set} If $M$ is an $\itpf1_2$ factor with
eigenvalue list $(\lambda_n, 1-\lambda_n)$ then the $T$-set is given by
the formula
$$
T(M)=\{t\in\R:\sum_{n=1}^\infty\big(1-|\lambda_n^{1+it}+(1-\lambda_n)^{1+it}|\big)<\infty\}.
$$
\end{lemma}

\bigskip

The following slightly abusive notation is convenient in this paper:
For a real $s\in\R$, write $s\imod {2\pi}$ for the unique element of
$$
\{s+2\pi p: p\in\Z\}\cap (-\pi,\pi].
$$
For $x\in c_0$ and $t\in\R$ define
$$
\delta_j^x(t)= tl_j^x\imod {2\pi}.
$$
When the value of $t$ is clear from the context we will usually
write $\delta_j^x$ for $\delta_j^x(t)$.

The next lemma is stated only for the family $(M_x)_{x\in c_0}$, but
is a special case of a well-known consequence of Lemma \ref{T-set}
which has been observed in many places in the literature (see e.g.
\cite{gioskan} and \cite{arwo68}). We include its proof for the sake
of completeness.

\begin{lemma}\label{criteria for T set} For each $t\in\R$, $t\in T(M_x)$ iff
$$
\sum_{j=1}^{\infty}N_je^{-l_j^x}(\delta_j^x(t))^2<\infty.
$$
\end{lemma}
\begin{proof}
By Lemma \ref{T-set},
\begin{align*}
t\in T(M_x) &\iff\sum_{j=1}^{\infty}N_j(1-|(\frac{1}{1+e^{-l_j^x}})^{1+it}+(\frac{e^{-l_j^x}}{1+e^{-l_j^x}})^{1+it}|)<\infty\\
&\iff\sum_{j=1}^{\infty}N_j(1-\frac{1}{1+e^{-l_j^x}}|1+e^{-l_j^x}e^{-il_j^xt}|)<\infty\\
&\iff\sum_{j=1}^{\infty}N_j(1-\frac{1}{1+e^{-l_j^x}}|1+e^{-l_j^x}e^{-i\delta_j^x(t)}|)<\infty\\
&\iff\sum_{j=1}^{\infty}N_je^{-l_j^x}(1-\cos(\delta_j^x(t)))<\infty\\
&\iff\sum_{j=1}^{\infty}N_je^{-l_j^x}((\delta_j^x(t))^2+O((\delta_j^x(t))^4))<\infty\\
&\iff\sum_{j=1}^{\infty}N_je^{-l_j^x}(\delta_j^x(t))^2<\infty
\label{sumconverges}
\end{align*}
\end{proof}

\begin{remark}\label{remark-t-set}
 The previous lemma sheds light on the motivation behind the
definition of the family $(M_x)_{x\in c_0}$. Indeed, since
$N_je^{-l_j^x} = 2^{j!(1-e^{x(j)/j!})}$ goes to $1$ when $j\to
\infty$, then to control the sum
$\sum_{j=1}^{\infty}N_je^{-l_j^x}(\delta_j^x(t))^2$, and thus the
$T$-set, it will be enough to control the size of $\delta_j^x(t)$.
This fact is what we will exploit in the next two lemmas.
\end{remark}

\begin{lemma}\label{nonempty}
For each $x\in c_0$,  $T(M_x)\neq \{0\}$.
\end{lemma}
\begin{proof}
Define
$$
t=\frac{1}{\ln 2}\sum_{j=1}^{\infty}\frac{a(j)}{j!e^{x(j)/j!}}
$$
where $a(j)\in (0,3\pi]$ is defined recursively by letting $a(1)=1$
and in general for $j>1$,
$$
a(j)=\left[-\sum_{k=1}^{j-1}j!e^{x(j)/j!}\frac{a(k)}{k!e^{x(k)/k!}}\right]
\imod {2\pi}+2\pi.
$$
Then $0<t<\infty$ and we have
\begin{align*} l_j^xt &=
\sum_{k=1}^{\infty}j!e^{x(j)/j!}\frac{a(k)}{k!e^{x(k)/k!}}\\
&= \sum_{k=1}^{j-1}j!e^{x(j)/j!}\frac{a(k)}{k!e^{x(k)/k!}} +a(j)+
\sum_{k=j+1}^{\infty}j!e^{x(j)/j!}\frac{a(k)}{k!e^{x(k)/k!}}
\end{align*}
and so
$$
\delta_j^x(t)=\sum_{k=j+1}^{\infty}j!e^{x(j)/j!}\frac{a(k)}{k!e^{x(k)/k!}}
\imod {2\pi}.
$$
If  $j$ is large enough that for all $k\geq j$ we have $1/2\leq
e^{x(k)/k!}\leq 2$ then
\begin{align*}
0\leq\sum_{k=j+1}^{\infty}j!e^{x(j)/j!}\frac{a(k)}{k!e^{x(k)/k!}}
 & \leq 12\pi\sum_{k=j+1}^{\infty}\frac{j!}{k!}\\
\tag{$\dag$} & \leq 12\pi\sum_{k=1}^{\infty}\frac{1}{(j+1)^k}\\
 & =\frac {12\pi} {j}
\end{align*}
Hence for $j$ sufficiently large it holds that
$$
\delta_j^x(t)=\sum_{k=j+1}^{\infty}j!e^{x(j)/j!}\frac{a(k)}{k!e^{x(k)/k!}}\sim \frac{1}{j}
$$
and so by $(\dag)$ and Lemma \ref{criteria for T set} we have $t\in
T(M_x)$.
\end{proof}

\begin{lemma}\label{comeagre}
For each $t\in \R\setminus \{0\}$ the set $\{x\in c_0: t\not \in
T(M_x)\}$ is a dense $G_\delta$ subset of $c_0$.
\end{lemma}
\begin{proof}
Since $T(M_x)$ is a subgroup of $(\R,+)$, we may assume that $t>0$.
For each $K\in\N$ let
$$
A_K=\{x\in c_0: (\exists L\in\N)\ \sum_{j=1}^L N_j
e^{-l_j^x}(\delta_j^x)^2>K\}.
$$
The set $A_K$ is open since for each $j\in\N$ the function
$x\mapsto(\delta^x_j)^2$ is continuous. By Lemma \ref{criteria for T
set} we have
$$
\{x\in c_0: t\not \in T(M_x)\}=\bigcap_{K\in\N} A_K
$$
so it suffices to show that $\bigcap_{K\in \N} A_K$ is dense. Let
$y\in c_0$ and $\varepsilon>0$. Pick $j_0\in \N$ such that for all
$j>j_0$,
$$
|y(j)|<\frac \varepsilon 2
$$
and
$$\frac {1}{t\ln(2)\sqrt{j_0}}<\frac \varepsilon 2\,.
$$
For $j\leq j_0$ define $x(j)=y(j)$. For $j>j_0$ define
$x(j)=j!\ln\left(1+\frac{a(j)}{t\ln(2)j!\sqrt{j}}\right)$, where
$a(j)$ is defined according to the following rule:
$$a(j)=
\begin{cases}
0 & \text{ if }\,\, |t\ln(2)j! \imod{2\pi}|\geq\frac{1}{2\sqrt{j}}\\
1 & \text{ if }\,\, |t\ln(2)j! \imod{2\pi}|<\frac{1}{2\sqrt{j}}.
\end{cases}
$$
It is clear that if $j>j_0$ then
$$
0\leq x(j)\leq
j!\frac{a(j)}{t\ln(2)j!\sqrt{j}}\leq\frac{1}{t\ln(2)\sqrt{j}}<\frac
{\varepsilon} 2,
$$
so $x\in c_0$ and $\|x-y\|_\infty<\varepsilon$. On the other hand we
have that
$$
t\l_j^x=t\ln(2)j!e^{x(j)/j!}=t\ln(2)j!\left(1+\frac{a(j)}{t\ln(2)j!\sqrt{j}}\right)=t\ln(2)j!+\frac{a(j)}{\sqrt
j}.
$$
By the choice of $a(j)$ we have $|\delta_j^x|=|tl_j^x\imod
{2\pi}|\geq\frac{1}{2\sqrt{j}}$. It follows that
$$
\sum_{j=1}^\infty
N_je^{-l_j^x}(\delta_j^x)^2\geq\sum_{j=j_0+1}^\infty
2^{j!(1-e^{x(j)/j!})}\left(\frac{1}{2\sqrt j}\right)^2=\infty,
$$
which shows that $x\in \bigcap_{K\in\N} A_K$.
\end{proof}

Recall that the equivalence relation $\sim_{\iso}$ in $c_0$ is
defined by $x\sim_{\iso} x'\iff M_x\simeq M_{x'}$. For $x\in c_0$
let $[x]_{\sim_{\iso}}=\{y\in c_0: y\sim_{\iso} x\}$.

\begin{lemma}
For each $x\in c_0$, $[x]_{\sim_{\iso}}$ is meagre.\label{meagre}
\end{lemma}
\begin{proof} By Lemma \ref{nonempty} there exists $t_0\in
T(M_x)\setminus\{0\}$. Then
$$
[x]_{\sim_{\iso}}\subseteq \{y\in c_0: t_0\in T(M_y)\}
$$
and so $[x]_{\sim_{\iso}}$ is meagre by Lemma \ref{comeagre}.
\end{proof}

\begin{remark}
It will be shown below that, in a precise way, $x\mapsto M_x$ is
Borel. By (the proof of) \cite[Theorem 2.2]{effros66} the
isomorphism relation $\simeq$ is analytic (see also \cite[Corollary
15]{sato09a}.) It follows that $\sim_{\iso}$ is analytic, and so by
the Kuratowski-Ulam Theorem \cite[8.41]{kechris95} we get from Lemma
\ref{meagre} that $\sim_{\iso}$ is meagre as a subset of $c_0\times
c_0$.
\end{remark}

We will need the following fact, a proof of which may be found in
\cite[Lemma 2.13]{arwo68}, see also \cite[Lemme
1.3.8]{connesthesis}.

\begin{prop}
If $M_1$ and $M_2$ are $\itpf1_2$ factors with eigenvalue lists
$(\lambda_{n,1},1-\lambda_{n,1})_{n\in\N}$ and
$(\lambda_{n,2},1-\lambda_{n,2})_{n\in\N}$, respectively, and
$$
\sum_{n=1}^\infty \left( (\lambda_{n,1})^{\frac 1
2}-(\lambda_{n,2})^{\frac 1 2} \right)^2+\left
((1-\lambda_{n,1})^{\frac 1 2}-(1-\lambda_{n,2})^{\frac 1
2}\right)^2<\infty,
$$
then $M_1$ and $M_2$ are unitarily isomorphic (and so they are
isomorphic.)\label{unitary}
\end{prop}

\begin{remark}
Denote by $c_{00}\subseteq c_0$ the set of all eventually zero
sequences. Then $c_{00}$ acts continuously on $c_0$ by addition. By
Proposition \ref{unitary} it follows easily that if $g\in c_{00}$
then
$$
M_{g+x}\simeq M_x
$$
for all $x\in c_0$. Thus the action of $c_{00}$ on $c_0$ preserves
$\sim_{\iso}$. Since $\sim_{\iso}$ is meagre in $c_0\times c_0$ and
clearly $c_{00}$-orbits are dense, we can now apply \cite[Theorem
3.4.5]{beckerkechris}, by which it follows that $E_0\leq_B
\sim_{\iso}$. Here $E_0$ denotes the equivalence relation in
$\{0,1\}^\N$ defined by
$$
x E_0 y\iff (\exists N)(\forall n\geq N) x(n)=y(n).
$$
Below we will show that the assignment $x\mapsto M_x$ is Borel, and
so it follows that $E_0$ is Borel reducible to isomorphism of
$\itpf1_2$ factors. Since $E_0$ is not smooth this provides a new
proof of the following:
\end{remark}

\noindent {\bf Theorem}\  (Woods, \cite{woods73}). \emph{ $E_0$ is Borel reducible to
isomorphism of $\itpf1_2$ factors. In particular
the isomorphism relation for $\itpf1_2$ factors is not smooth.}

\medskip

Arguably the proof exhibited here is simpler than the argument given in \cite{woods73}, partly because we
avoid to construct an explicit Borel reduction from $E_0$ to isomorphism of $\itpf1_2$ factors that made
Woods' original proof quite involved.
Observe that since $c_{00}$ doesn't admit a Polish group structure,
Hjorth's theory of turbulence does not apply to its actions. In what
follows we will overcome this difficulty by defining a group $G$
that can play the role of $c_{00}$, but which is also Polish.
Specifically, consider the set
$$
G=\{a\in c_0: \sum_{j=1}^\infty 2^{j!}a(j)^2<\infty \}.
$$
The set $G$ becomes a separable real Hilbert space when equipped
with the inner product given by $\langle
a,b\rangle=\sum_{j=1}^\infty 2^{j!}a(j)b(j)$. Since $c_{00}\subset
G$, it follows that $G$ is dense in $c_0$. Moreover, since $G\neq
c_0$, by Lemma \ref{turbulence} the action of $G$ on $c_0$ by
addition is turbulent, has meagre classes and all the classes are
dense. The following lemma shows that the $G$-action on $c_0$
preserves the $\sim_{\iso}$ classes:
\begin{lemma}\label{preserves orbits}
If $a\in G$, then $M_x$ is unitarily equivalent to $M_{a+x}$. In
particular, $x\sim_{\iso} (a+x)$.
 \end{lemma}
 \begin{proof}
By Proposition \ref{unitary} it is enough to check that the sum
$$\sum_{j=1}^\infty N_j\Big\{\Big [\Big(\frac{1}{1+e^{-l_j^x}}\Big)^{\frac 1 2}-\Big(\frac{1}{1+e^{-l_j^{a+x}}}\Big)^{\frac 1 2}\Big]^2+
\Big[\Big(\frac{e^{-l_j^x}}{1+e^{-l_j^x}}\Big)^{\frac 1
2}-\Big(\frac{e^{-l_j^{a+x}}}{1+e^{-l_j^{a+x}}}\Big)^{\frac 1
2}\Big]^2\Big\}$$ is finite. Since the derivatives of the functions
$f(s)=(1+e^{-s})^{-\frac 1 2}$ and
$h(s)=\big(\frac{e^{-s}}{1+e^{-s}}\big)^{\frac 1 2}$ are bounded by
1 whenever $s>0$, the previous sum is bounded by
\begin{align*}
\sum_{j=1}^\infty 2N_j[l_j^x- l_j^{a+x}]^2 &=
 2\ln^2(2)\sum_{j=1}^\infty 2^{j!}(j!)^2e^{2x(j)/j!}[1-e^{a(j)/j!}]^2\\
 &< K\sum_{j=1}^\infty 2^{j!}(j!)^2[(a(j)/j!)^2+O(a(j)/j!)^3)]\\
&< \tilde K\sum_{j=1}^\infty 2^{j!}a(j)^2
\end{align*}
for appropriate constants $K$ and $\tilde K$, and this is finite
whenever $a\in G$.
\end{proof}

\begin{theorem} The equivalence relation $\sim_{\iso}$ is
generically $S_\infty$-ergodic, and $\sim_{\iso}$ is not
classifiable by countable structures.\label{sinftyergodic}
\end{theorem}
\begin{proof}
Let $G$ be as above, and let $E_G^{c_0}$ denote the orbit
equivalence relation induced by the action of $G$ on $c_0$. Then by
Lemma \ref{preserves orbits} we have $E_G^{c_0}\subseteq
\sim_{\iso}$. Since $G$ acts turbulently, it follows by
\cite[Theorem 3.18]{hjorth00} that $E_G^{c_0}$ is generically
$S_\infty$-ergodic, and so as noted in the discussion of
$S_\infty$-ergodicity in \S 2, $\sim_{\iso}$ is generically
$S_\infty$-ergodic.

Suppose, seeking a contradiction, that $S_\infty$ acts continuously on the
Polish space $Y$ and
$$
\sim_{\iso}\leq_B E_{S_\infty}^Y.
$$
If $f:c_0\to Y$ were a Borel reduction witnessing this then $f$
would map a comeagre set in $c_0$ to the same $S_\infty$-class. But
this would contradict that all $\sim_{\iso}$ classes are meagre by
Lemma \ref{meagre}, and so $f$ can't be a reduction. Hence
$\sim_{\iso}$ is not classifiable by countable structures.
\end{proof}

\begin{remark}
It follows from Theorem \ref{sinftyergodic} that the set
$$
\{x\in c_0: T(M_x) \text{ is uncountable}\}
$$
is comeagre in $c_0$, since otherwise the assignment $x\mapsto
T(M_x)$ would give an $\sim_{\iso}$-invariant assignment of
countable subsets of $\R$ on a comeagre set, and so by \cite[Lemma
3.14]{hjorth00} the function $x\mapsto T(M_x)$ would be constant on
a comeagre set. But this contradicts Lemma \ref{nonempty} and
\ref{comeagre}.

It follows from the above and Lemma \ref{nonempty} that
$$
\{x\in c_0: M_x\text{ is of type } \III_0\}
$$
is comeagre, as it should be, since $\itpf1$ factors of type
$\III_\lambda$, $0<\lambda\leq 1$, are classified by a single real
number $\lambda$. It actually follows from \cite[Proposition
1.3]{gioskan} that for all $x\in c_0$, $T(M_x)$ is uncountable, thus
$M_x$ is of type $\III_0$, but using an entirely different line of
argument.
\end{remark}

Recall from \cite{sato09a} and \cite{sato09b} that if $\mathcal H$
is a separable complex Hilbert space, then $\vN(\mathcal H)$ denotes
the standard Borel space of von Neumann algebras acting on $\mathcal
H$, equipped with the Effros Borel structure originally introduced
in \cite{effros65} and \cite{effros66}. Let $\simeq^{\vN(\mathcal
H)}$ denote the isomorphism relation in $\vN(\mathcal H)$.

\begin{theorem}
The isomorphism relation for $\itpf1_2$ factors is not classifiable
by countable structures.\label{mainthm2}
\end{theorem}
\begin{proof}
It suffices to show that there is a Borel function
$f:c_0\to\vN(\ell^2(\N))$ such that for all $x\in c_0$ we have
$f(x)\simeq M_x$, since then by Theorem \ref{sinftyergodic} it
follows that $\sim_{\iso}\leq_B \simeq^{\vN(\ell^2(\N))}$. That such
a function $f$ exists follows from the next three lemmas.
\end{proof}

\begin{lemma}
Suppose $X$ is a standard Borel space and $(H_x:x\in X)$ is a family
of infinite dimensional separable Hilbert spaces, and that
$(e_n^x)_{n\in\N}$ is an orthonormal basis of $H_x$ for each $x\in
X$. Suppose further that $Y$ is a standard Borel space space and
$(T_y^x:x\in X,y\in Y)$ is a family of operators such that
$T_y^x\in\mathcal B(H_x)$ for all $y\in Y$, $x\in X$ and that the
functions
$$
X\times Y\to\C: (x,y)\mapsto \langle T^x_y e_n^x, e_m^x\rangle
$$
are Borel for all $n,m$. Then there is a Borel function
$\theta:X\times Y\to\mathcal B(\ell^2(\N))$ and a family
$(\varphi_x: x\in X)$ such that
\begin{enumerate}
\item $\varphi_x\in \mathcal B(H_x, \ell^2(\N))$ satisfies $\varphi_x(e_n^x)=e_n$, where $(e_n)_{n\in\N}$ is the standard basis for $\ell^2(\N)$.
\item For all $x\in X$, $y\in Y$ and $\xi\in H_x$ we have $\theta(x,y)(\varphi_x(\xi))=\varphi_x(T_y^x(\xi))$
\end{enumerate}
Moreover, if $M_x$ is the von Neumann algebra generated by the
family $(T_y^x:y\in Y)$, and there are Borel functions
$$
\psi_n: X\to Y
$$
such that $(T_{\psi_n(x)}^x:n\in\N)$ generates $M_x$ for each $x\in
X$, then there is a Borel function $\hat\theta:X\to \vN(\ell^2(\N))$
such that $\hat\theta(x)\simeq M_x$ for all $x\in
X$.\label{codinglemma}
\end{lemma}

\begin{proof}
The family $(\varphi_x:x\in X)$ is uniquely defined by (1), and
$\theta$ is uniquely defined by
$$
\theta(x,y)=T\iff (\forall n,m) \langle
Te_n,e_m\rangle_{\ell^2(\N)}=\langle T_y^x e_n^x,
e_m^x\rangle_{H_x},
$$
which also gives a Borel definition of the graph of $\theta$, so
$\theta$ is Borel by \cite[14.12]{kechris95} since $\mathcal B(\ell^2(\N))$ is a standard Borel space when given the Borel structure generated by the weak topology. If we let
$$
f_n:X\to\vN(\ell^2(\N)): x\mapsto \theta(x,\psi_n(x)).
$$
then the ``moreover'' part follows from \cite[Theorem 2]{effros65}
since $M_x$ is isomorphic to
$$
\{f_n(x): n\in\N\}''\in\vN(\ell^2(\N)).
$$
\end{proof}

\begin{lemma}
There is a Borel function $f: (0,1)^\N\to\vN(\ell^2(\N))$ such that
for all $x\in (0,1)^\N$, $f(x)$ is isomorphic to
$$
N_x=\bigotimes_{n=1}^\infty (M_2(\C), (x(n), 1-x(n)),
$$
the $\itpf1$ factor with eigenvalue list $(x(n),1-x(n))_{n\in\N}$.
\end{lemma}

\begin{proof}
Let $M_2(\C)$ act on itself by multiplication. Then let
$\eta:(0,1)^\N\to (M_2(\C)^4)^\N$ be a Borel function such that
$$
\eta(x)(n)_1=\left (\begin{array}{cc}
\sqrt{x(n)} & 0\\
0 & \sqrt{1-x(n)}
\end{array}\right)
$$
and $\{\eta(x)(n)_i: i\in\{1,2,3,4\}\}$ is an orthonormal basis for
$M_2(\C)$. For each $\vec{i}\in \{1,2,3,4\}^\N$ such that
$\vec{i}(k)=1$ eventually, let
$$
e_{\vec{i}}^x= \eta(x)(1)_{\vec{i}(1)}\otimes
\eta(x)(2)_{\vec{i}(2)}\otimes\cdots\otimes
\eta(x)(n)_{\vec{i}(n)}\otimes\cdots
$$
Then $(e_{\vec{i}}^x: \vec{i}(k)=1 \text{ eventually})$ is an
orthonormal basis for the Hilbert space
$$
H_x=\bigotimes_{n=1}^\infty (M_2(\C), \eta(x)(n)_1).
$$
Let $M_2(\C)^{<\N}$ denote the set of sequences in $\vec a\in
M_2(\C)^\N$ that $\vec a(k)=I$ eventually. Then for all $\vec i,\vec
j$ the map $(0,1)^\N\times M_2(\C)^{<\N}\to\C$ given by
$$
(x,\vec a)\mapsto \langle\vec a(1)\eta(x)(1)_{\vec{i}(1)}\otimes
\vec a(2)\eta(x)(2)_{\vec{i}(2)}\otimes\cdots,\eta(x)(1)_{\vec
j(1)}\otimes\eta(x)(2)_{\vec j(2)}\otimes\cdots\rangle
$$
is continuous. Since  for each $x$ fixed, $M_2(\Q[i ])^{<\N}$
generates $N_x$ as a von Neumann algebra, the Lemma now follows from
Lemma \ref{codinglemma}.
\end{proof}

\begin{lemma}
There is a Borel function $f:c_0\to\vN(\ell^2(\N))$ such that
$f(x)\simeq M_x$.
\end{lemma}
\begin{proof}
Immediate by the previous lemma, since
$$
x\mapsto \left(\frac 1 {1+e^{-l_j^x}}:j\in\N\right)
$$
is continuous.
\end{proof}

\begin{remarks}

1. Consider Cantor space $\{0,1\}^\N$ and the odometer action of
$\Z$ on $X=\{0,1\}^\N$ (i.e. ``adding one with carry''.) Let $z\in
(0,1)^\N$, and let $\mu^z$ be the product measure
$$
\mu^z=\prod_{n=1}^\infty (z(n)\delta_0+(1-z(n))\delta_1),
$$
where $\delta_0$ and $\delta_1$ denote the Dirac measures on
$\{0,1\}$ concentrating on $0$ and $1$, respectively. The measure
class of $\mu^z$ is preserved by the odometer action, and if $
\mu^z$ is ergodic for the odometer we let
$N_z=L^\infty(X,\mu^z)\rtimes \Z$ be the Krieger factor obtained
from the group-measure space construction. Then
$$
N_z=L^\infty(X,\mu^z)\rtimes \Z\simeq\bigotimes_{n=1}^\infty
(M_2(\C), (z(n), 1-z(n)),
$$
see \cite[III.3.2.18]{blackadar06}. By Krieger's celebrated Theorem
(\cite[8.4]{krieger76}, see also \cite[III.3.2.19]{blackadar06}) the
group measure space factors $N_z$ and $N_{z'}$ are isomorphic
precisely when the corresponding measure class-preserving odometer
actions are orbit equivalent. If we now, for each $x\in c_0$, let
$z_x(n)=\lambda_n^x$, where $(\lambda_n^x, 1-\lambda_n^x)_{n\in\N}$
is the eigenvalue list of the factor $M_x$, then since all the
factors $M_x$ are type $\III$, the measure $\mu^{z_x}$ is non-atomic
and ergodic for the odometer. Thus we obtain the following
consequence of Theorem \ref{mainthm2}:
\begin{theorem}
The odometer actions of $\Z$ on $\{0,1\}^\N$ that preserve the
measure class of some ergodic non-atomic $\mu^z$ as above, are not
classifiable up to orbit equivalence by countable structures.
\end{theorem}

2. The observation made in \cite[Corollary 8]{sato09a} is equally
pertinent to the main result of this paper: Since the proof relies
only on Baire category techniques, Theorem \ref{mainthm2} shows that
it is not possible to construct in Zermelo-Fraenkel set theory
without the Axiom of Choice a function that completely classifies
$\itpf1_2$ factors up to isomorphism by assigning countable
structures type invariants.

\medskip

3. In \cite{sato09b} we asked (Problem 4) if all possible $K_\sigma$
subgroups of $\R$ appear as the $T$-set of some $\itpf1$ factor.
Stefaan Vaes has kindly pointed out to us that this is already known
{\it not} to be the case: This follows from the results of \cite[\S
2]{homepa91}, see also \cite[\S 2]{pova09}.
\end{remarks}

\bibliographystyle{amsplain}
\bibliography{turbITPFI}

\end{document}